\documentclass[a4paper]{article}

\usepackage{graphicx,amssymb}
\usepackage{amsmath}
\usepackage{cases}
\usepackage{float}
\usepackage{listings}
\usepackage{times}
\usepackage{tikz}
\usepackage{epic}
\usepackage{titlesec}
\usepackage{geometry}
\usepackage{amsthm}
\usepackage{bm}
\usepackage{stmaryrd}
\usepackage{subfigure}
\usepackage{indentfirst}
\usepackage{xcolor}
\usepackage{epstopdf}
\usepackage{enumerate}
\usepackage{enumitem}
\usepackage{hyperref}

\definecolor{red}{rgb}{1.00,0.00,0.00}
{\numberwithin{equation}{section}
\setlength{\parindent}{1em}

\newtheorem{theorem}{Theorem}[section]
\newtheorem{lemma}{Lemma}[section]
\newtheorem{remark}{Remark}[section]

\newtheorem{example}{Example}[section]

\SetSymbolFont{stmry}{bold}{U}{stmry}{m}{n}
\renewcommand{\footnotesize}{\scriptsize}
\newcommand{\normmm}[1]{{\left\vert\kern-0.25ex\left\vert
\kern-0.25ex\left\vert #1
    \right\vert\kern-0.25ex\right\vert\kern-0.25ex\right\vert}}
\geometry{left=3cm,right=3cm,top=4cm,bottom=2.5cm}

\newcommand{\draft}[1]{\textcolor{black}{#1}}
\newcommand{\bs}{\boldsymbol}
\newcommand{\matr}[1]{#1}

\begin{document}
\title{A new staggered DG method for the Brinkman problem robust in the Darcy and Stokes limits}
\author{Lina Zhao\footnotemark[1]\qquad
\;Eric Chung\footnotemark[2]\qquad
\;Ming Fai Lam\footnotemark[3]}
\renewcommand{\thefootnote}{\fnsymbol{footnote}}
\footnotetext[1]{Department of Mathematics, The Chinese University of Hong Kong, Shatin, New Territories, Hong Kong SAR, China. ({lzhao@math.cuhk.edu.hk})}
\footnotetext[2]{Department of Mathematics, The Chinese University of Hong Kong, Shatin, New Territories, Hong Kong SAR, China. ({tschung@math.cuhk.edu.hk})}
\footnotetext[3]{Department of Mathematics, The Chinese University of Hong Kong, Shatin, New Territories, Hong Kong SAR, China. ({mflam@math.cuhk.edu.hk})}

\date{}
\maketitle

\textbf{Abstract:}
In this paper we propose a novel staggered discontinuous Galerkin method for the Brinkman problem on general quadrilateral and polygonal meshes. The proposed method is robust in the Stokes and Darcy limits, in addition, hanging nodes can be automatically incorporated in the construction of the method, which are desirable features in practical applications. There are three unknowns involved in our formulation, namely velocity gradient, velocity and pressure. Unlike the original staggered DG formulation proposed for the Stokes equations in \cite{KimChung13}, we relax the tangential continuity of velocity and enforce different staggered continuity properties for the three unknowns, which is tailored to yield an optimal $L^2$ error estimates for velocity gradient, velocity and pressure independent of the viscosity coefficient. Moreover, by choosing suitable projection, superconvergence can be proved for $L^2$ error of velocity. Finally, several numerical results illustrating the good performances of the proposed method and confirming the theoretical findings are presented.


\pagestyle{myheadings} \thispagestyle{plain}
\markboth{L. Zhao} {A new SDG method for \draft{Brinkman} problem}

\section{Introduction}

In this paper, we consider the following \draft{Brinkman} problem on a bounded polygonal domain $\Omega \subset \mathbb{R}^2$
\begin{equation}
\begin{split}
-\epsilon \Delta \bm{u}+\alpha \bm{u}+\nabla p& =\bm{f}\;\;\; \mbox{in}\;\Omega,\\
\text{div}\,\bm{u}&=g\quad \mbox{in}\;\Omega,\\
\bm{u}&=0 \quad \mbox{on }\;\partial \Omega,
\end{split}
\label{eq:model}
\end{equation}
where $\bm{u}$ is the velocity, $p$ is the pressure, $\epsilon$ is the effective viscosity constant and $\bm{f}\in L^2(\Omega)^2$ is the external body force. In addition, we also assume that there exist real numbers $\alpha_{\text{min}}$ and $\alpha_{\text{max}}$ such that $0<\alpha_{\text{min}}\leq \alpha\leq \alpha_{\text{max}}$.

The Brinkman problem \eqref{eq:model} is used to model fluid motion in porous media with fractures. It locally behaves likes a Stokes or Darcy problem depending on the value of a dimensionless parameter, which can be interpreted as a local friction coefficient. Developing a numerical method for the Binkman problem that is robust for both Stokes and Darcy limits is challenging. The numerical experiments in \cite{MardalTaiWinther02,Hannukainen11} indicate that the convergence rate deteriorates as the Brinkman becomes Darcy-dominating when certain stable Stokes elements are used; such elements include the conforming $P_2$-$P_0$ element, the nonconforming Crouzeix-Raviart element, the Mini element and the Taylor-Hood elements. Similarly, the convergence rate deteriorates as the Brinkman problem becomes Stokes-dominating when Darcy stable elements such as the lowest order Raviart-Thomas elements \cite{MardalTaiWinther02} are used. To remedy this and obtain uniformly stable methods for both Darcy and Stokes limits, a large deal of effort has been made. The use of Darcy-tailored, $H(\text{div};\Omega)$-conforming finite element method is studied in \cite{Konno11}, where a dual mixed formulation is exploited and a symmetric interior penalty Galerkin (SIPG) method is employed to enforce the tangential continuity of the velocity. The use of Stokes-tailored is analyzed in \cite{MardalTaiWinther02}, where a new nonconforming element is constructed. In addition to the aforementioned works, one can also refer to \cite{BurmanHansbo05,BurmanHansbo07,XieXu08,Badia09,Braack11,Guzman12,Konno12,EvansHughes13,
Gatica14,Vassilevski14,Anaya15ESAIM,Anaya16,Benner16,Hong16,Araya17,FuQiu19} and the references therein for a glance at uniformly stable methods for the Brinkman problem.
Recently, some novel methods have been proposed to solve the Brinkman problem on general quadrilateral and polygonal meshes, which is tricky and usually requires special treatment in order to deliver robust results with respect to the rough grids. The development of numerical methods for Brinkman problem on general meshes is still in its infancy and the existing methods that have been successfully designed on general meshes for Brinkman problem include the weak Galerkin method, the virtual element method and the hybrid high-order methods \cite{MuWangYe14,Caceres17,Vacca18,Botti18}.

Staggered discontinuous Galerkin method (SDG) is initially proposed for wave propagation problems in \cite{EricEngquistwave,ChungWave2} on triangular meshes, since then it has been successfully applied to a wide range of partial differential equations arising from practical applications \cite{ChungCiarletYu13,LeeKim16,KimChungLam16,ChungQiu17,DuChung18,
ChungParkZhao18,CheungChungKim18}. Recently SDG methods on triangular meshes have been extended to solve Darcy law and the Stokes equations on general quadrilateral and polygonal meshes \cite{LinaPark, LinaParkShin}. The application to coupled Stokes-Darcy problem is also considered in \cite{LinaParkSD}.  The principal idea behind SDG method for the Stokes equations proposed in \cite{LinaParkShin} is to decompose the computational domain into the primal partition (quadrilateral or polygonal meshes), then the primal mesh is further divided into the union of triangles by connecting a certain interior point to the vertices of the primal mesh, and the dual mesh is generated simultaneously during this subdivision process. Based on the primal meshes and the dual meshes, we can construct three sets of basis functions with staggered continuity properties to approximate velocity gradient, velocity and pressure, respectively. SDG methods earn many desirable features, which can be summarized as follows: First, it preserves mass conservations and superconvergence can be achieved; Second, it can be flexibly applied to general quadrilateral and polygonal meshes and hanging nodes can be simply incorporated in the construction of the method, which favors adaptive mesh refinement; Third, SDG method can be extended to high order polynomial approximations straightforwardly. Fourth, thanks to the staggered continuity properties involved, numerical flux or penalty term is not needed. The aforementioned properties make SDG method highly desirable in practice.


In the staggered DG method proposed in \cite{KimChung13,LinaParkShin} for the Stokes equations, the discrete unknown for velocity gradient is continuous in the normal direction over the dual edges, the discrete unknown for velocity is continuous over the primal edges, and the discrete unknown for pressure is continuous over the dual edges. Our undisplayed analysis and numerical experiments indicate that if we apply the original staggered DG method for the Stokes equations to the Brinkman problem, then it will lead to poor performances when the Brinkman problem becomes Darcy-dominating. This can be briefly explained as follows: when the Brinkman problem becomes Darcy-dominating, i.e., $\epsilon\rightarrow 0$, the finite element pairs for velocity and pressure is no longer stable in terms of Darcy law. In order to deliver uniformly stable results, we need to modify the finite element spaces designed for the discrete unknowns. More precisely, we relax the tangential continuity of velocity and switch the staggered continuity properties for all the variables involved. Moreover, the normal continuity of the velocity gradient over the dual edges is enforced in a weak sense. Then based on the modified spaces, we design a novel staggered DG method for the Brinkman problem that is robust in both Stokes and Darcy limits. We can view our novel method as Darcy-tailored staggered DG method for the Brinkman problem. A rigorous error analysis for velocity gradient, velocity and pressure measured in $L^2$ error is presented, where the error estimates are shown to be independent of the viscosity coefficient. In addition, superconvergent error estimates in $L^2$ error of velocity with suitable projection are obtained. We emphasize that thanks to the staggered continuity properties involved in our discrete finite element spaces, no numerical fluxes or stabilization terms are required, which is appreciated compared to other DG methods. In addition, our method can handle fairly general meshes and hanging nodes can be automatically incorporated in the construction of the method, which is highly appreciated for the practical applications. Furthermore, our formulation is based on velocity gradient-velocity-pressure, existing methods based on this formulation is still quite rare \cite{FuQiu19}.

The rest of the paper can be organized as follows. In the next section, we derive our novel SDG method for the Brinkman problem. Rigorous error analysis for velocity gradient, velocity and pressure is carried out in section~\ref{sec:error}. Then we present some numerical experiments to verify the robustness and accuracy of the proposed method in section~\ref{sec:numerical}. Finally, a conclusion is given at the end of the paper.

\section{Description of a new SDG method}

In this section, we attempt to derive a uniformly stable SDG method for the Brinkman problem. The key ingredient is to design appropriate finite element spaces earning staggered continuity properties. Based on which we can design a novel SDG method which is robust both in Stokes and Darcy limits, besides it can be flexibly applied to general quadrilateral and polygonal meshes.

We introduce an additional unknown $L = \sqrt{\epsilon}\nabla \bm{u}$, thereby the \draft{Brinkman} model problem \eqref{eq:model} can be recast into the following first order system
\begin{align}
L&=\sqrt{\epsilon}\nabla \bm{u}\quad \mbox{in}\;\Omega,\label{eq:model1}\\
- \sqrt{\epsilon}\text{div}\, L+\alpha \bm{u}+\nabla p& =\bm{f}\hspace{1.1cm} \mbox{in}\;\Omega,\label{eq:model2}\\
\text{div}\, \bm{u}&=g\hspace{1.2cm} \mbox{in}\;\Omega,\label{eq:model3}\\
\bm{u}&=\bm{0} \hspace{1.2cm} \mbox{on }\;\partial \Omega
\end{align}
with the restriction $\int_\Omega p\;dx=0$.

We introduce some notations that will be used later. For a set $D\subset \mathbb{R}^{2}$, we denote the scalar product in $L^{2}(D)$ by $(\cdot,\cdot)_{D}$, namely $(p,q)_{D}:=\int_{D} p\,q \;dx$, we use the same symbol $(\cdot,\cdot)_{D}$ for the inner product in $[L^{2}(D)]^{2}$ and in $[L^{2}(D)]^{2\times 2}$. More precisely $(\sigma,\tau)_{D}:
=\sum_{i=1}^{2}\sum_{j=1}^{2}(\sigma^{ij},
\tau^{ij})_{D}$ for $\sigma,\tau\in [L^2(D)]^{2\times 2}$.
We denote by $(\cdot, \cdot)_e$ the scalar product in $L^2(e), e\subset \mathbb{R}$ (or duality pairing), for a scalar, vector, or tensor functions.
Automatically, we define $\|\cdot\|_{0, D}$ to be the $L^2$-norm
on $D$ and $\|\cdot\|_{0,e}$ the $L^2$-norm on $e$. When $D$
coincides with $\Omega$, the subscript $\Omega$ will be dropped
unless otherwise mentioned. Throughout this paper, we use $C$ to
denote a generic positive constant which may have a different
value at different occurrences.

Now we introduce the construction of our SDG spaces, in line with this we then present our novel SDG method.
 To begin, we construct three meshes: the primal mesh $\mathcal{T}_{u}$, the dual mesh $\mathcal{T}_{d}$, and the primal simplicial submeshes $\mathcal{T}_h$. For a polygonal domain $\Omega$, consider a general mesh $\mathcal{T}_{u}$ (of $\Omega$) that consists of nonempty connected close disjoint subsets of $\Omega$:
\begin{align*}
\bar{\Omega}=\bigcup_{T\in \mathcal{T}_{u}}T.
\end{align*}
We let $\mathcal{F}_{u}$ be the set of all primal edges in this partition and $\mathcal{F}_{u}^{0}$ be the
subset of all interior edges, that is, the set of edges in $\mathcal{F}_{u}$ that do not lie on $\partial\Omega$. We construct the primal submeshes $\mathcal{T}_h$ as a triangular subgrid of the primal grid: for an element $T\in \mathcal{T}_{u}$, elements of $\mathcal{T}_h$ are obtained by connecting the interior point $\nu$ to all vertices of $\mathcal{T}_{u}$ (see Figure~\ref{grid}).
Moreover, we will use $\mathcal{F}_{p}$ to denote the set of all the dual edges generated by this subdivision process. For each triangle
$\tau\in \mathcal{T}_h$, we let $h_\tau$ be the diameter of
$\tau$, $h_e$ be the length of edge $e\subset\partial \tau$, and $h=\max\{h_\tau, \tau\in \mathcal{T}_h\}$.
In addition, we define $\mathcal{F}:=\mathcal{F}_{u}\cup \mathcal{F}_{p}$ and $\mathcal{F}^{0}:=\mathcal{F}_{u}^{0}\cup \mathcal{F}_{p}$. We rename the primal element by $S(\nu)$, which is assumed to be star-shaped with respect to a ball of radius $\rho h_{S(\nu)}$, where $\rho$ is a positive constant. In addition, we assume that for every edge $e\in \partial S(\nu)$, it satisfies $h_e\geq Ch_{S(\nu)}$. More discussions about mesh regularity assumptions for general meshes can be referred to \cite{Beir13,Pietro15,Cangiani16}.
The construction for general meshes is illustrated in Figure~\ref{grid},
where the black solid lines are edges in $\mathcal{F}_{u}$
and red dotted lines are edges in $\mathcal{F}_{p}$.

Finally, we construct the dual mesh. For each interior edge $e\in \mathcal{F}_{u}^0$, we use $D(e)$ to denote the dual mesh, which is the union of the two triangles in $\mathcal{T}_h$ sharing the edge $e$,
and for each boundary edge $e\in\mathcal{F}_{u}\backslash\mathcal{F}_{u}^0$, we use $D(e)$ to denote the triangle in $\mathcal{T}_h$ having the edge $e$,
see Figure~\ref{grid}.

For each edge $e$, we define
a unit normal vector $\bm{n}_{e}$ as follows: If $e\in \mathcal{F}\setminus \mathcal{F}^{0}$, then
$\bm{n}_{e}$ is the unit normal vector of $e$ pointing towards the outside of $\Omega$. If $e\in \mathcal{F}^{0}$, an
interior edge, we then fix $\bm{n}_{e}$ as one of the two possible unit normal vectors on $e$.
When there is no ambiguity,
we use $\bm{n}$ instead of $\bm{n}_{e}$ to simplify the notation. In addition, we use $\bm{t}$ to denote the corresponding unit tangent vector. We also introduce some notations that will be employed throughout this paper. Let $k\geq 0$ be the order of approximation. For every $\tau
\in \mathcal{T}_{h}$ and $e\in\mathcal{F}$,
we define $P^{k}(\tau)$ and $P^{k}(e)$ as the spaces of polynomials of degree less than or equal to $k$ on $\tau$ and $e$, respectively. In the following, we use $\nabla_h$ and $\text{div}_h$ to denote the element-wise gradient and divergence operators, respectively.


We now define jump terms which will be used throughout the paper.
For each triangle $\tau_i$ in $\draft{\mathcal{T}_h}$ such that $e \subset \partial \tau_i$, we let $\bs{n}_i$ be the outward unit normal vector on $e\subset \partial \tau_i$. The sign $\delta_i$ of $\bs{n}_i$ with respect to $\bs{n}$ on $e$ is then given by
\begin{equation*}
\delta_i = \bs{n}_i \cdot \bs{n} = \begin{cases}
1 & \text{ if } \bs{n}_i = \bs{n} \text{ on } e,\\
-1 & \text{ if } \bs{n}_i = -\bs{n} \text{ on } e.
\end{cases}
\end{equation*}
For a double-valued scalar quantity $\phi$, let $\phi_i$ be the value of $\phi \vert_{\tau_i}$ restricted on $e$. The jump $[\phi]$ across an edge $e$ can then be defined as:
\begin{equation*}
[\phi] = \delta_1 \phi_1 + \delta_2 \phi_2.
\end{equation*}
Similarly, for a vector quantity $\bs{\phi}$ and a matrix quantity $\matr{\Phi}$, the jumps $[\bs{\phi} \cdot \bs{n}]$ and $[\matr{\Phi} \bs{n}]$ across an edge $e$ is defined as:
\begin{equation}
\begin{split}
[\bs{\phi} \cdot \bs{n}] & = \delta_1 (\bs{\phi}_1 \cdot \bs{n}) + \delta_2 (\bs{\phi}_2 \cdot \bs{n}),\\
[\matr{\Phi} \bs{n}] & = \delta_1 (\matr{\Phi}_1 \bs{n}) + \delta_2 (\matr{\Phi}_2 \bs{n}).
\label{eq:jump2}
\end{split}
\end{equation}

\begin{figure}
\centering
\includegraphics[width=12cm]{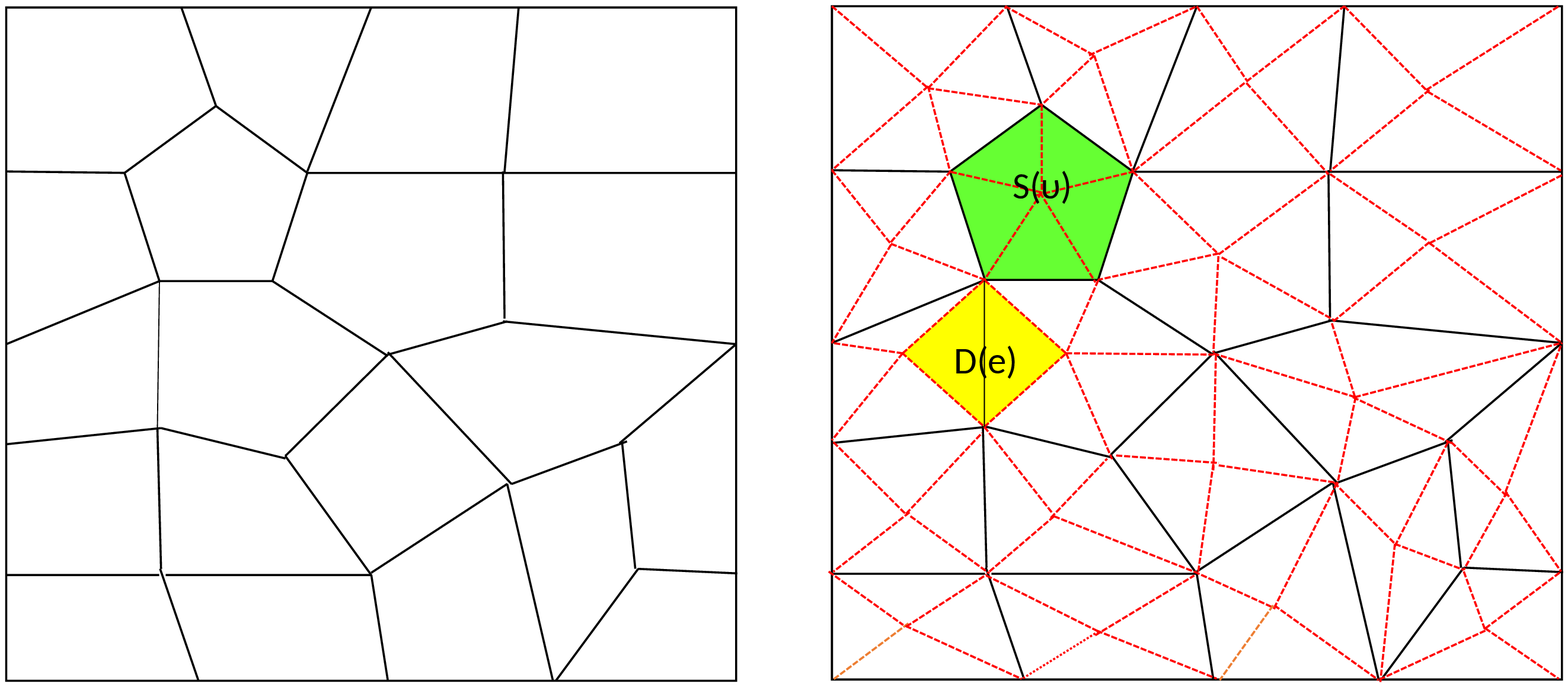}
\setlength{\abovecaptionskip}{-0.5cm}
\caption{Schematic of the primal mesh $S(\nu)$, the dual mesh $D(e)$ and the primal simplicial submeshes.}
\label{grid}
\end{figure}

To specify our new SDG method, we now introduce the finite element spaces employed. Note that the spaces employed in this paper are different from those exploited in original SDG method (cf. \cite{ChungWave2,LinaPark,LinaParkShin}). Our undisplayed analysis and numerical experiments indicate that the original SDG method fails to deliver uniformly stable results. The modifications made in this paper can overcome this issue.
We first define
the following finite element space for velocity:
\begin{equation*}
U^h=\{\bs{v} \: : \:  \bs{v} \vert_{\tau} \in P^k(\tau)^2;  \tau \in \draft{\mathcal{T}_h}; \bs{v} \cdot \bs{n} \text{ is continuous over } e \in \mathcal{F}_p\},
\end{equation*}
which is equipped by
\begin{align*}
\|\bm{v}\|_{X_1}^2&=\|\bm{v}\|_0^2+\sum_{e\in \mathcal{F}_p}h_e\|\bm{v}\cdot \bm{n}\|_{0,e}^2,\\
\|\bm{v}\|_{Z_1}^2&=\|\text{div}_h\bm{v}\|_0^2+\sum_{e\in \mathcal{F}_u}h_e^{-1}\|[\bm{v}\cdot \bm{n}]\|_{0,e}^2,\\
\|\bm{v}\|_{Z_2}^2&=\|\nabla_h \bm{v}\|_0^2+\sum_{e\in \mathcal{F}_u}h_e^{-1}\|[\bm{v}]\|_{0,e}^2+\sum_{e\in \mathcal{F}_p}h_e^{-1}\|[(\bm{v}\cdot \bm{t})\bm{t}]\|_{0,e}^2.
\end{align*}

Next, we define the following finite element space for velocity gradient:
\begin{equation*}
W^h=\{\matr{G} \: : \:  \matr{G} \vert_{\tau} \in P^k(\tau)^{2\times 2};  \tau \in \draft{\mathcal{T}_h}; \matr{G} \bs{n} \text{ is continuous over } e\in \mathcal{F}_u^0\}.
\end{equation*}
In this space, we define the corresponding discrete $H(\text{div};\Omega)$ semi-norm
\begin{align*}
\|G\|_{Z'}^2&=\|\text{div}_h G\|_0^2+\sum_{e\in \mathcal{F}_p}h_e^{-1}\|[G\bm{n}]\|_{0,e}^2.
\end{align*}

Then, we define the following locally $H^{1}(\Omega)$-conforming finite element space for pressure:
\begin{equation*}
P^h=\{q \: : \:  q \vert_{\tau} \in P^k(\tau);  \tau \in \draft{\mathcal{T}_h}; q \text{ is continuous over } e\in \mathcal{F}_u^0;\int_{\Omega} q \;dx=0\}
\end{equation*}
with the corresponding discrete $L^2$ norm and $H^1$ semi-norm
\begin{align*}
\normmm{q}_{0,h}^2&=\|q\|_0^2+\sum_{e\in \mathcal{F}_u}h_e\|q\|_{0,e}^2,\\
\normmm{q}_{1,h}^2&=\|\nabla_h q\|_0^2+\sum_{e\in \mathcal{F}_p}h_e^{-1}\|[q]\|_{0,e}^2.
\end{align*}

Finally, we define the following space which is employed to enforce the weak continuity of the velocity gradient over the dual edge
\begin{equation*}
\widehat{U}^h=\{\widehat{\bs{v}} \: : \:  \widehat{\bs{v}} \vert _e \in P^k(e)^2; \widehat{\bs{v}} \cdot \bs{n}\mid_e = 0 \quad\forall e \in \mathcal{F}_p \}.
\end{equation*}

We are now ready to derive a new SDG method for \eqref{eq:model1}-\eqref{eq:model3}, which is robust in terms of $\epsilon$ and performs efficiently on fairly general meshes. To begin, multiplying \eqref{eq:model1} by a test function $\matr{G} \in W^h$ and integrating over $D(e)$ for $e \in \mathcal{F}_u$, in addition, the tangential component $(\bs{u} \cdot \bs{t}) \bs{t}$ of $\bs{u}$ on the edges $e \in \mathcal{F}_p$ is approximated by an additional unknown $\widehat{\bs{u}}$, thereby we can obtain
\begin{equation*}
\begin{split}
\int_{D(e)} \matr{L}:\matr{G}\; dx
& =
 - \sqrt{\epsilon}\int_{D(e)} \bs{u} \cdot (\text{div } \matr{G}) \; dx
+ \sqrt{\epsilon}\int_{\partial D(e)\cap\mathcal{F}_p} \widehat{\bs{u}} \cdot(\matr{G}\bs{n})\; ds\\
&\qquad+ \sqrt{\epsilon} \int_{\partial D(e)\cap\mathcal{F}_p} (\bs{u} \cdot \bs{n}) \bs{n} \cdot(\matr{G}\bs{n})\; ds,
\end{split}
\end{equation*}
where the facts that $[G\bm{n}]\mid_e=0$ for $e\in \mathcal{F}_u$ and $\bm{u}=\bm{0}$ on $\partial \Omega$ are utilized.

Next, we multiply \draft{\eqref{eq:model2}} by a test function $\bs{v} \in U^h$, integrate over $\mathcal{S}(\nu)$ for $\nu \in \mathcal{N}$, exploit the fact that $[\bm{v}\cdot \bm{n}]=0$ for $e\in \mathcal{F}_p$ and obtain
\begin{equation*}
\begin{split}
\int_{\mathcal{S}(\nu)} \bs{f}\cdot \bs{v} \; dx
& = \sqrt{\epsilon}\int_{\mathcal{S}(\nu)}\matr{L}:\nabla \bs{v}\; dx- \sqrt{\epsilon}\int_{\partial \mathcal{S}(\nu)\cap\mathcal{F}_u} (\matr{L}\bs{n}) \cdot \bs{v}  \; ds 
- \sqrt{\epsilon}\sum_{\tau\in S(\nu)}\int_{\partial \tau\cap \mathcal{F}_p} \left[(\matr{L}\bs{n}) \cdot (\bs{v} \cdot \bs{t}) \bs{t} \right]  \; ds\\
& \qquad + \int_{\mathcal{S}(\nu)}\alpha\bs{u} \cdot \bs{v}\;dx -\int_{\mathcal{S}(\nu)} p \, (\text{div } \bs{v})\;dx+\int_{\partial \mathcal{S}(\nu)} p \, (\bs{v}\cdot \bs{n})\;ds,
\end{split}
\end{equation*}
where, in accordance with our definitions of jumps in \eqref{eq:jump2}, the term $\left[(\matr{L}\bs{n}) \cdot (\bs{v} \cdot \bs{t}) \bs{t} \right]$ on $e \in \mathcal{F}_p$ should be formally written as
\begin{equation*}
\left[(\matr{L}\bs{n}) \cdot (\bs{v} \cdot \bs{t}) \bs{t} \right] = \left[(\bs{v}^T\matr{L}) \cdot \bs{n}\right] - (\bs{v} \cdot \bs{n}) \bs{n} \cdot \left[\matr{L}\bs{n}\right].
\end{equation*}
Multiplying \eqref{eq:model3} by a test function $q \in P^h$ and integrating on $D(e)$ for $e \in \mathcal{F}_u$, performing integration by parts with the fact that $[q]\vert_e = 0$ for $e \in \mathcal{F}_u^0$, we deduce
\begin{equation*}
\begin{split}
\int_{D(e)} g \; q \;dx
=\int_{D(e)} (\text{div }\bs{u}) \, q \;dx
=-\int_{D(e)}\bs{u} \cdot \nabla q \;dx+\int_{\partial D(e)\cap\mathcal{F}_p}(\bs{u} \cdot \bs{n}) \, q \;ds.
\end{split}
\end{equation*}
Finally, we enforce normal continuity of $\matr{L}$ weakly on $\mathcal{F}_p$. For any $\widehat{\bs{v}} \in \widehat{U}^h$ and any $e \in \mathcal{F}_p$, we have
\begin{equation*}
 \int_{e} [\matr{L}\bs{n}] \cdot \widehat{\bs{v}}  \;ds = 0.
\end{equation*}

Combining the preceding derivations, now we can formulate our novel SDG formulation for the Brinkman problem \eqref{eq:model}: Find $(L_h,\bm{u}_h,\widehat{\bm{u}}_h,p_h)\in W^h\times U^h\times \widehat{U}^h\times P^h$ such that
\begin{equation}
\begin{split}
(L_h,G)&=\sqrt{\epsilon}B_h^*(\bm{u}_h,G)+\sqrt{\epsilon}T_h^*(\widehat{\bm{u}}_h,G)\quad \forall G\in W^h,\\ \sqrt{\epsilon}B_h(L_h,\bm{v})+(\alpha\bm{u}_h,\bm{v})+b_h^*(p_h,\bm{v})&=(\bm{f},\bm{v})\quad \forall \bm{v}\in U^h,\\
\draft{-}b_h(\bm{u}_h,q)&=(g,q)\quad \forall q\in P^h,\\
T_h(L_h,\widehat{\bm{v}})&=0\quad \forall \widehat{\bm{v}}\in \widehat{U}^h,
\end{split}
\label{eq:SDG-discrete}
\end{equation}
where the bilinear forms are defined by
\begin{align*}
B_h^*(\bm{v},G)&=-\int_\Omega \bm{v}\cdot\text{div}_h \matr{G}\;dx+\sum_{e\in \mathcal{F}_p}\int_e(\bm{v}\cdot \bm{n})\bm{n}\cdot [G\bm{n}]\;ds,\\
B_h(G,\bm{v})&=\int_\Omega G\cdot \nabla_h \bm{v}\;dx-\sum_{e\in \mathcal{F}_u}\int_e[\bm{v}]\cdot( G\bm{n})\;ds-\sum_{e\in \mathcal{F}_p}\int_e [(\bm{v}\cdot \bm{t})\bm{t}\cdot (G\bm{n})]\;ds,\\
T_h^*(\widehat{\bm{v}},G)&=\sum_{e\in \mathcal{F}_p}\int_e \widehat{\bm{v}}\cdot [G\bm{n}]\;ds,\\
T_h(G,\widehat{\bm{v}})&=\sum_{e\in \mathcal{F}_p}\int_e [G\bm{n}]\cdot \widehat{\bm{v}}\;ds,\\
b_h^*(q,\bm{v})&=-\int_\Omega q\,\text{div}_h \bm{v}\;dx+\sum_{e\in \mathcal{F}_u}\int_e q[\bm{v}\cdot\bm{n}]\;ds,\\
b_h(\bm{v},q)&=\int_\Omega \bm{v}\cdot\nabla_h q\;dx-\sum_{e\in \mathcal{F}_p}\int_e\bm{v}\cdot\bm{n}[q]\;ds.
\end{align*}
Performing integration by parts reveals the following adjoint properties
\begin{equation}
\begin{split}
B_h(G,\bm{v})&=B_h^*(\bm{v},G)\quad\, \forall (G,\bm{v})\in W^h\times U^h,\\
b_h(\bm{v},q)&=b_h^*(q,\bm{v})\qquad \forall (\bm{v},q)\in U^h\times P^h,\\
T_h(G,\widehat{\bm{v}})&=T_h^*(\widehat{\bm{v}},G)\hspace{0.6cm} \forall (G,\draft{\widehat{\bm{v}}})\in W^h\times \draft{\widehat{U}^h}.
\end{split}
\label{eq:adjoint}
\end{equation}

To facilitate the analysis, we define the subspace of $W^h$ by
\begin{align*}
\widehat{W}^h:=\{G\in W^h: \int_e \draft{[G\bm{n}] \cdot \hat{\bm{v}}}\;ds=0\quad \forall \hat{\bm{v}}\in \widehat{U}^h\draft{, \forall e\in\mathcal{F}_p}\}.
\end{align*}
The degrees of freedom for $\widehat{W}^h$ are given by

(WD1). For edge $e_1\in \mathcal{F}_u$
\begin{align}
\phi_{e_1}(G):= \int_{e_1} (G\bm{n})\draft{\cdot}\,p_k\;ds\quad \forall p_k\in P^k(e)^2\label{eq:dof1}
\end{align}
and for edge $e_2\in \mathcal{F}_p$
\begin{align}
\phi_{e_2}(G):= \int_{e_2} (G\bm{n})\cdot \bm{t}\;p_k\;ds\quad \forall p_k\in P^k(e)\label{eq:dof2}.
\end{align}
(WD2). For each $\tau\in \mathcal{T}_h$, we have
\begin{align*}
\phi_\tau(G):=\int_\tau G\draft{:}\, p_{k-1}\;dx\quad \forall p_{k-1}\in [P^{k-1}(\tau)]^{2\times 2}.
\end{align*}

\begin{lemma}
Any function $G\in \widehat{W}^h$ is uniquely determined by the degrees of freedom (WD1)-(WD2).
\end{lemma}

\begin{proof}
The dimension of the space $P^k(e)$ is $k+1$ while the dimension of the space $P^k(\tau)$ is $(k+1)(k+2)/2$. First we have
\begin{align*}
\mbox{dim}(\widehat{W}^h)=2|\mathcal{T}_h|(k+1)(k+2)-2|\mathcal{F}_u^0|(k+1)
-|\mathcal{F}_p|(k+1),
\end{align*}
where the subtraction of the second and the third term results in the continuity of functions in $\widehat{W}^h$ on each edge $e\in \mathcal{F}_u^0$ and $e\in \mathcal{F}_p$, respectively. Let $|WD|$ be the total number of degrees of freedom associated with (WD1) and (WD2). Then we have
\begin{align*}
|WD|=2|\mathcal{F}_u|(k+1)+|\mathcal{F}_p|(k+1)+2|\mathcal{T}_h|k(k+1),
\end{align*}
where the terms on the right hand side denotes the number of degrees of freedom associated with (WD1) and (WD2), respectively. Subtracting, we have
\begin{align*}
\mbox{dim}(\widehat{W}^h)-|WD|
=2(k+1)(2|\mathcal{T}_h|-2|\mathcal{F}_u^0|-|\mathcal{F}_u\backslash\mathcal{F}_u^0|-|\mathcal{F}_p|).
\end{align*}
Note that we can associate each edge in $\mathcal{F}_u^0$ to two triangle in $\mathcal{T}_h$, and associate each edge in $\mathcal{F}_u\backslash \mathcal{F}_u^0$ to one triangle, and  associate each edge in $\mathcal{F}_p$ to one triangle in $\mathcal{T}_h$. Thus, we have $2|\mathcal{F}_u^0|+|\mathcal{F}_u\backslash\mathcal{F}_u^0|+|\mathcal{F}_p|=2|\mathcal{T}_h|$. Hence, we have $\mbox{dim}(\draft{\widehat{W}^h})=|WD|$.

Since $\mbox{dim}(\widehat{W}^h)=|WD|$, it suffices to show uniqueness. Since $G\in \widehat{W}^h$ is defined such that all degrees of freedom associated with both (WD1) and (WD2) are equal to zero. That is
\begin{align}
\phi_{e_1}(G)&=0\quad \forall e_1\in \mathcal{F}_u,\\
\phi_{e_2}(G)& =0\quad \forall e_2\in \mathcal{F}_p,\\
\phi_\tau(G)&=0\quad \forall \tau\in \mathcal{T}_h\label{degreeG3}.
\end{align}
Let $\bm{n}^{(i)},i=1,\cdots,3$ denote the three unit normal vectors of three edges of each triangle $\tau\in \mathcal{T}_h$, where $\bm{n}^{(1)}$ denotes the unit normal vector for $e\in \partial \tau\cap \mathcal{F}_u$ and the remaining two associate to the edges in $\partial \tau\cap \mathcal{F}_p$ with counterclockwise orientation. Write $\bm{n}^{(i)}=(n^{(i)}_1,n^{(i)}_2)^T$.
Then we can take by following \eqref{eq:dof1} and \eqref{eq:dof2}
\begin{align*}
G\bm{n}^{(1)}=(\lambda_{\tau,1}q_{k-1}^{(1)}, \lambda_{\tau,1}q_{k-1}^{(2)})^T
\end{align*}
for $e\in \mathcal{F}_u$ and
\begin{align*}
(G\bm{n}^{(2)})\cdot \bm{t}^{(2)}=\lambda_{\tau,2}q_{k-1}^{(3)}, \quad (G\bm{n}^{(3)})\cdot \bm{t}^{(3)}=\draft{\lambda_{\tau,3}}q_{k-1}^{(4)}
\end{align*}
for $e\in \mathcal{F}_p$. \draft{Here, $q_{k-1}^{(j)},j=1,\cdots,4$, are to be determined. And $\lambda_{\tau,i},i=1,\cdots,3$, are linear functions defined on $\tau$ such that $\lambda_{\tau,i}=1$ at vertex $i$ and vanishes at the remaining vertices.} Let $A$ be the matrix such that
\begin{align*}
A=\left(
    \begin{array}{cccc}
      n^{(1)}_1 & n^{(1)}_2   & 0       & 0 \\
      0       & 0           & n^{(1)}_1 & n^{(1)}_2 \\
      t^{(2)}_1n^{(2)}_1 & t^{(2)}_1n^{(2)}_2 & t^{(2)}_2n^{(2)}_1 & t^{(2)}_2n^{(2)}_2 \\
       t^{(3)}_1n^{(3)}_1& t^{(3)}_1n^{(3)}_2  & t^{(3)}_2n^{(3)}_1 & t^{(3)}_2n^{(3)}_2 \\
    \end{array}
  \right).
\end{align*}
We can check that $A$ is invertible. So we can write (here we rewrite
$G=\left(
     \begin{array}{cc}
       G_{11} & G_{12} \\
       G_{21} & G_{22} \\
     \end{array}
   \right)
$ in a collum-wise way, i.e., $G=(G_{11},G_{12},G_{21},G_{22})^T$)
\begin{align*}
AG =\left(
        \begin{array}{c}
          \lambda_{\tau,1}q_{k-1}^{(1)}\\
          \lambda_{\tau,1}q_{k-1}^{(2)}\\
          \lambda_{\tau,2}q_{k-1}^{(3)}\\
          \lambda_{\tau,3}q_{k-1}^{(4)} \\
        \end{array}
      \right),\quad G = A^{-1} \left(
        \begin{array}{c}
          \lambda_{\tau,1}q_{k-1}^{(1)} \\
          \lambda_{\tau,1}q_{k-1}^{(2)} \\
          \lambda_{\tau,2}q_{k-1}^{(3)} \\
          \lambda_{\tau,3}q_{k-1}^{(4)} \\
        \end{array}
      \right).
\end{align*}
Then we also take $p_{k-1}$ in a collum-wise way as below
\begin{align*}
p_{k-1}=A^{-1}\left(
                 \begin{array}{cc}
                   q_{k-1}^{(1)} \\
                    q_{k-1}^{(2)} \\
                   q_{k-1}^{(3)} \\
                    q_{k-1}^{(4)} \\
                 \end{array}
               \right).
\end{align*}
It follows from \eqref{degreeG3} that
\begin{align*}
0=(G, p_{k-1})_\tau&=\int_\tau \left(
        \begin{array}{c}
          \lambda_{\tau,1}q_{k-1}^{(1)}\\
          \lambda_{\tau,1}q_{k-1}^{(2)}\\
          \lambda_{\tau,2}q_{k-1}^{(3)}\\
          \lambda_{\tau,3}q_{k-1}^{(4)} \\
        \end{array}
      \right)^TA^{-T}A^{-1}\left(
                 \begin{array}{cc}
                   q_{k-1}^{(1)} \\
                    q_{k-1}^{(2)} \\
                   q_{k-1}^{(3)} \\
                    q_{k-1}^{(4)} \\
                 \end{array}
               \right)\;dx\\
               &\geq C \int_\tau \left(
                 \begin{array}{cc}
                   q_{k-1}^{(1)} \\
                    q_{k-1}^{(2)} \\
                   q_{k-1}^{(3)} \\
                    q_{k-1}^{(4)} \\
                 \end{array}
               \right)^{T} A^{-T}A^{-1}\left(
                 \begin{array}{cc}
                   q_{k-1}^{(1)} \\
                    q_{k-1}^{(2)} \\
                   q_{k-1}^{(3)} \\
                    q_{k-1}^{(4)} \\
                 \end{array}
               \right)\;dx
=C \int_\tau |p_{k-1}|^2\;dx.
\end{align*}
Consequently, we have $q_{k-1}^{(j)}=0$ for $j=1,\cdots,4$. Thus $G=0$ on $\tau\in \mathcal{T}_h$.
\end{proof}

Based on the definition of $\widehat{W}^h$ and the discrete formulation \eqref{eq:SDG-discrete}, we can conclude that $L\in \widehat{W}^h$. Therefore, we can reformulate our discrete formulation \eqref{eq:SDG-discrete} and obtain the following equivalent formulation: find $\draft{(L_h,\bm{u}_h,p_h)\in \widehat{W}^h\times U^h\times P^h}$ such that
\begin{equation}
\begin{split}
(L_h,G)&=\sqrt{\epsilon}B_h^*(\bm{u}_h,G)\quad \forall G\in \widehat{W}^h, \\ \sqrt{\epsilon}B_h(L_h,\bm{v})+(\alpha\draft{\bm{u}_h},\bm{v})+b_h^*(p_h,\bm{v})&=(\bm{f},\bm{v})\quad \forall \bm{v}\in U^h,\\
b_h(\bm{u}_h,q)&=(g,q)\quad \forall q\in P^h.
\end{split}
\label{eq:SDG-discrete-new}
\end{equation}
\begin{remark}
{\rm The current method is different from original SDG method proposed in \cite{ChungWave2,LinaPark,LinaParkShin}, where the velocity is continuous over the primal edges. Our (undisplayed) numerical results indicate that the original SDG methods fails to deliver stable results when the Brinkman problem becomes Darcy dominating. To remedy this, we relax the tangential continuity of the velocity, in addition, we also make simultaneous modifications for the velocity gradient in order to guarantee the well-posedness of the proposed method.}

\end{remark}
Before closing this section, we provide some basic ingredients which will be employed in the subsequent analysis. First, the following inf-sup condition holds as introduced in \cite{ChungWave2}
\begin{align}
\inf_{q\in P^h}\sup_{\bm{v}\in U^h}\frac{b_h(\bm{v},q)}{\|\bm{v}\|_{0}\normmm{q}_{1,h}}\geq C.\label{eq:infsupbh}
\end{align}
Next, we introduce the following discrete $L^2$ norm for $\widehat{W}^h$
\begin{align*}
\|G\|_{X'}^2& = \|G\|_0^2+\sum_{e\in \mathcal{F}_u}h_e\|G\bm{n}\|_{0,e}^2+\sum_{e\in \mathcal{F}_p} h_e\|(G\bm{n})\cdot \bm{t}\|_{0,e}^2\quad \forall G\in \widehat{W}^h.
\end{align*}
Scaling arguments imply the following norm equivalentce
\begin{align}
\|G\|_0\leq \|G\|_{X'}\leq C \|G\|_0.\label{eq:scalingG}
\end{align}

\section{Error analysis}\label{sec:error}
In this section, we will analyze the unique solvability of the discrete system \eqref{eq:SDG-discrete} and prove the convergence estimates. To this end, we begin with proving the following inf-sup condition.



\begin{theorem}
There exists a uniform constant $C>0$ such that
\begin{align}
\inf_{\bm{v}\in U^h}\sup_{G\in \widehat{W}^h}\frac{B_h(G,\bm{v})}{\|G\|_{X'}\|\bm{v}\|_{Z_2}}\geq C.\label{eq:inf-supB}
\end{align}

\end{theorem}

\begin{proof}

Since $G\in \widehat{W}^h$, $B_h$ can be rewritten as
\begin{align*}
B_h(G,\bm{v})&=\int_\Omega G:\nabla \bm{v}_h\;dx-\sum_{e\in \mathcal{F}_u}\int_e[\bm{v}]\cdot( G\bm{n})\;ds-\sum_{e\in \mathcal{F}_p}\int_e [(\bm{v}\cdot \bm{t})](G\bm{n})\cdot\bm{t}\;ds.
\end{align*}
It suffices to find $G\in \widehat{W}^h$ such that
\begin{align*}
B_h(G,\bm{v})\geq C \|\bm{v}\|_{Z_2}^2\quad \mbox{and} \quad \|G\|_{X'}\leq C \|\bm{v}\|_{Z_2}.
\end{align*}
We define $G\in \widehat{W}^h$ corresponding to the degrees of freedom (WD1)-(WD2) such that
\begin{align*}
(G,\nabla \bm{v})_\tau&=\|\nabla \bm{v}\|_{0,\tau}^2\quad \draft{\forall} \tau\in \mathcal{T}_h,\\
-(G\bm{n}, [\bm{v}])_e&=h_e^{-1}\|[\bm{v}]\|_{0,e}^2\quad \forall e\in \mathcal{F}_u,\\
-((G\bm{n})\cdot \bm{t},[\bm{v}\cdot \bm{t}])_e &=h_e^{-1}\|[(\bm{v}\cdot\bm{t})\bm{t}]\|_{0,e}^2\quad \forall e\in \mathcal{F}_p.
\end{align*}
We can obtain from the definition of $B_h(\cdot,\cdot)$ that
\begin{align*}
B_h(G,\bm{v})= \|\bm{v}\|_{Z_2}^2.
\end{align*}
In addition, an application of scaling argument implies that
\begin{align*}
\|G\|_{X'}\leq C \|\bm{v}\|_{Z_2}.
\end{align*}
Therefore, the proof is complete.
\end{proof}

\begin{theorem}(existence and uniqueness).
The system \eqref{eq:SDG-discrete} admits a unique solution.

\end{theorem}

\begin{proof}

\eqref{eq:SDG-discrete} is a square linear system, uniqueness implies existence, thus it suffices to show uniqueness. Let $\bm{f}=0$ and $g=0$ in \eqref{eq:SDG-discrete} and taking $G = L_h, \bm{v}=\bm{u}_h,q=p_h,\widehat{\bm{v}}=\widehat{\bm{u}}_h$ and adding the resulting equations yield
\begin{align*}
\|L_h\|_0^2+\|\alpha\bm{u}_h\|_0^2=0.
\end{align*}
Thus, $L=0$ and $\bm{u}_h=\bm{0}$. Consequently, the second equation of \eqref{eq:SDG-discrete} is reduced to
\begin{align*}
b_h^*(p_h,\bm{v})=0\quad \forall \bm{v}\in U^h.
\end{align*}
In view of inf-sup condition \eqref{eq:infsupbh} and the discrete Poincar\'{e} inequality (cf. \cite{Brenner03}), we can obtain
\begin{align*}
\|p_h\|_0\leq C\normmm{p_h}_{1,h}&\leq C \sup_{\bm{v}\in U^h}\frac{b_h(\bm{v},p_h)}{\|\bm{v}\|_{0}} = C \sup_{\bm{v}\in U^h}\frac{b_h^*(p_h,\bm{v})}{\|\bm{v}\|_{0}}=0.
\end{align*}
Therefore, $p_h=0$. Finally, we have
\begin{align*}
\draft{T_h^*}(\widehat{\bm{u}}_h,G)=0.
\end{align*}
We can take $[G\bm{n}]\mid_e=\widehat{\bm{u}}_h\in \widehat{U}^h,e\in \mathcal{F}_p$. Therefore, $\widehat{\bm{u}}_h=0$. The assertion follows.

\end{proof}

The inf-sup condition \eqref{eq:inf-supB} implies the existence of an interpolation operator $\Pi_h:\draft{[H^1(\Omega)]^{2\times2}}\rightarrow \widehat{W}^h$ such that
\begin{align}
B_h(L-\Pi_hL,\bm{v})=0\quad \forall \bm{v}\in U^h.\label{eq:BhPi}
\end{align}
Also, \eqref{eq:inf-supB} is equivalent to the following
\begin{align*}
\inf_{G\in \widehat{W}^h}\sup_{\bm{v}\in U^h}\frac{B_h(G,\bm{v})}{\|G\|_{X'}\|\bm{v}\|_{Z_2}}\geq C.
\end{align*}
Consequently, we have for any polynomial $p_k\in [P^k]^{2\times 2}$
\begin{align*}
\|p_k-\Pi_hp_k\|_{X'}\leq C \sup_{\bm{v}\in U^h}\frac{B_h(p_k-\Pi_hp_k,\bm{v})}{\|\bm{v}\|_{Z_2}}=0.
\end{align*}
Thus, $p_k=\Pi_hp_k$. By the standard theory for polynomial preserving operators and the trace inequalies (cf. \cite{Ciarlet78,ChungWave2}), we obtain
\begin{align}
\|L-\Pi_hL\|_0&\leq C h^{k+1}|L|_{k+1},\label{eq:Lerror}\\
\|L-\Pi_h L\|_{Z'}&\leq C h^k |L|_{k+1}.\label{eq:Lerror2}
\end{align}
To facilitate later analysis, we also need to define the following two projection operators.
Let $I_h: H^1(\Omega)\rightarrow P^h$ be defined by
\begin{align*}
	(I_h q-q,\phi)_e
		&=0 \quad \forall \phi\in P^k(e),\draft{\forall} e\in \mathcal{F}_{u},\\
	(I_hq-q,\phi)_\tau
		&=0\quad \forall \phi\in P^{k-1}(\tau),\draft{\forall} \tau\in \mathcal{T}_h
\end{align*}
and $J_h: \draft{H^1(\Omega)^2}\rightarrow U^h$ be defined by
\begin{align*}
	((J_h\bm{v}-\bm{v})\cdot\bm{n},\varphi)_e
		&=0\quad \forall \varphi\in P^{k}(e),\ \forall e\in \mathcal{F}_{p},\\
	(J_h\bm{v}-\bm{v}, \bm{\phi})_\tau
		&=0\quad \forall \bm{\phi}\in P^{k-1}(\tau)^2,\ \forall \tau\in \mathcal{T}_h.
\end{align*}
It is easy to see that $I_h$ and $J_h$ are well defined polynomial preserving operators. In addition, the following approximation properties hold for $q\in H^{k+1}(\Omega)$ and $\bm{v}\in H^{k+1}(\Omega)^2$ (cf. \cite{Ciarlet78,ChungWave2})
\begin{align}
\normmm{q-I_hq}_{0,h}&\leq C h^{k+1}|q|_{k+1},\label{eq:Perror}\\
\|\bm{v}-J_h\bm{v}\|_{X_1}&\leq C h^{k+1}|\bm{v}|_{k+1}\label{eq:uerror},\\\
\|\bm{v}-J_h\bm{v}\|_{Z_1}&\leq Ch^k |\bm{v}|_{k+1}\label{eq:uerrorZ1},\\
\|\bm{v}-J_h\bm{v}\|_{Z_2}&\leq Ch^k |\bm{v}|_{k+1}\label{eq:uerrorZ}.
\end{align}
Furthermore, the following equations are satisfied
\begin{align}
b_h(\bm{v}-J_h\bm{v},q)&=0\quad \forall q\in P^h, \label{eq:bhJ}\\
b_h^*(q-I_hq,\bm{v})&=0\quad \forall \bm{v}\in U^h. \label{eq:bhI}
\end{align}
Also, we define $\pi_h\draft{:L^2(\Omega)^2\rightarrow\mathbb{R}^2}$ by
\begin{align*}
\pi_h \bm{\phi}\mid_\tau = \frac{1}{|\tau|}\int_\tau \bm{\phi}\;dx,
\end{align*}
which satisfies
\begin{align}
\|\bm{\phi}-\pi_h\bm{\phi}\|_{0,\tau}\leq C h_\tau \|\nabla \bm{\phi}\|_{0,\tau}\quad \forall \bm{\phi}\in H^1(\tau)^2.\label{eq:pih}
\end{align}

The following lemma will also be needed in later analysis.
\begin{lemma}\label{lem:estimate}
The following estimates hold \draft{ for any $\bm{v}\in H^1_0(\Omega)^2$}
\begin{align}
\|J_h\bm{v}\|_0&\leq C\|\bm{v}\|_1,\quad \|\nabla (J_h\bm{v})\|_0\leq C \|\bm{v}\|_1,\label{eq:Jh1}\\
\Big(\sum_{e\in \mathcal{F}_u}h_e^{-1}\|[J_h\bm{v}]\|_{0,e}^2\Big)^{1/2}&\leq C  \|\bm{v}\|_1,\quad \Big(\sum_{e\in \mathcal{F}_p}h_e^{-1}\|[(J_h\bm{v}\cdot \bm{t})\bm{t}]\|_{0,e}^2\Big)^{1/2}\leq C \|\bm{v}\|_1.\label{eq:Jh2}
\end{align}

\end{lemma}

\begin{proof}

One can refer to \cite{ChungWave2} for the proof of the first inequality of \eqref{eq:Jh1} and we omit it for simplicity. Triangle inequality and \eqref{eq:uerrorZ} yield
\begin{align*}
\|\nabla (J_h\bm{v})\|_0\leq (\|\nabla(J_h\bm{v}-\bm{v})\|_0+\|\nabla \bm{v}\|_0)\leq C \|\bm{v}\|_1.
\end{align*}
Trace inequality, \eqref{eq:uerror} and \eqref{eq:uerrorZ} imply
\begin{align*}
h_e^{-1}\|[J_h\bm{v}]\|_{0,e}^2&=h_e^{-1}\|[\bm{v}-J_h\bm{v}]\|_{0,e}^2\leq C \sum_{\tau\in D(e)} (h_\tau^{-2}\|\bm{v}-J_h\bm{v}\|_{0,\tau}^2+\|\nabla (\bm{v}-J_h\bm{v})\|_{0,\tau}^2)\\
&\leq C \|\bm{v}\|_{1,D(e)}^2.
\end{align*}
Summing up over all the edges in $\mathcal{F}_u$ leads to the first inequality of \eqref{eq:Jh2}.

Proceeding analogously, we can also prove the second inequality of \eqref{eq:Jh2}, which is omitted for simplicity.
Therefore, the proof is complete by combining the preceding arguments.
\end{proof}

\begin{theorem}\label{thm:L2}
Assume that $(L,\bm{u},p)\in [H^{k+1}(\Omega)]^{2\times 2}\times H^{k+1}(\Omega)^{2}\times H^{k+1}(\Omega)$ and $(L_h,\bm{u}_h,p_h)\in W^h\times U^h\times P^h$ is the numerical solution of \eqref{eq:SDG-discrete}, then there exists a positive constant $C>0$ independent of the meshsize and the coefficients $\alpha,\epsilon$ such that
\begin{align*}
\|L-L_h\|_0+\|\alpha^{1/2}(\bm{u}-\bm{u}_h)\|_0&\leq Ch^{k+1}(\|L\|_{k+1}+\alpha_{\text{max}}^{1/2}\|\bm{u}\|_{k+1}),\\
\sqrt{\epsilon}(\|\bm{u}-\bm{u}_h\|_{Z_2}+\|\bm{u}-\bm{u}_h\|_{Z_1})&\leq Ch^k (\alpha_{\text{max}}^{1/2}\|\bm{u}\|_{k+1}+(1+\sqrt{\epsilon})\|L\|_{k+1}),\\
\|p-p_h\|_0&\leq Ch^{k+1}(\|p\|_{k+1}+(1+\sqrt{\epsilon})\|L\|_{k+1}+\alpha_{\text{max}}\|\bm{u}\|_{k+1}), k\geq 1.
\end{align*}

\end{theorem}

\begin{proof}
Performing integration by parts on \eqref{eq:SDG-discrete-new}, we can obtain the following error equations
\begin{equation}
\begin{split}
(L-L_h,G)&=\sqrt{\epsilon}B_h^*(\bm{u}-\bm{u}_h,G)\quad \forall G\in \widehat{W}^h, \\ \sqrt{\epsilon}B_h(L-L_h,\bm{v})+(\alpha(\bm{u}-\bm{u}_h),\bm{v})+b_h^*(p-p_h,\bm{v})&=0\quad \forall \bm{v}\in U^h,\\
b_h(\bm{u}-\bm{u}_h,q)&=0\quad \forall q\in P^h.
\end{split}
\label{eq:error}
\end{equation}
Taking $G= \Pi_hL-L_h, \bm{v}=J_h\bm{u}-\bm{u}_h, q=I_hp-p_h$ in the above equations, we can obtain
\begin{align*}
(L-L_h,\Pi_hL-L_h)-\sqrt{\epsilon}B_h^*(\bm{u}-\bm{u}_h, \Pi_hL-L_h)=0,\\
\sqrt{\epsilon}B_h(L-L_h,J_h\bm{u}-\bm{u}_h)+(\alpha(\bm{u}-\bm{u}_h),J_h\bm{u}-\bm{u}_h)
+b_h^*(p-p_h,J_h\bm{u}-\bm{u}_h)&=0,\\
b_h(\bm{u}-\bm{u}_h,I_hp-p_h)&=0.
\end{align*}
Adding the above equations and exploiting the adjoint properties \eqref{eq:adjoint} yield
\begin{align*}
(L-L_h,\Pi_hL-L_h)+(\alpha(\bm{u}-\bm{u}_h),J_h\bm{u}-\bm{u}_h)=0.
\end{align*}
Therefore,
\begin{align*}
\|\Pi_h L-L_h\|_0^2+\|\alpha^{1/2}(J_h\bm{u}-\bm{u}_h)\|_0^2&=(\Pi_hL-L,\Pi_hL-L_h)
+(\alpha(J_h\bm{u}-\bm{u}),J_h\bm{u}-\bm{u}_h)\\
&\leq \|\Pi_hL-L\|_0\|\Pi_hL-L_h\|_0\\
&\;+\|\alpha^{1/2}(J_h\bm{u}-\bm{u})\|_0\|\alpha^{1/2}(J_h\bm{u}-\bm{u}_h)\|_0.
\end{align*}
Young's inequality implies
\begin{align*}
\|\Pi_h L-L_h\|_0+\|\alpha(J_h\bm{u}-\bm{u}_h)\|_0\leq C \Big(\|\Pi_hL-L\|_0+\alpha_{\text{max}}^{1/2}\|J_h\bm{u}-\bm{u}\|_0\Big).
\end{align*}
On the other hand, we have from inf-sup condition \eqref{eq:inf-supB} and \eqref{eq:error}
\begin{align*}
\sqrt{\epsilon}\|J_h\bm{u}-\bm{u}_h\|_{Z_2}\leq C \sup_{G\in \widehat{W}^h}\frac{\sqrt{\epsilon}B_h(G,J_h\bm{u}-\bm{u}_h)}{\|G\|_0}=C \sup_{G\in \widehat{W}^h}\frac{(L-L_h,G)}{\|G\|_0}\leq C \|L-L_h\|_0.
\end{align*}
The estimation of $\|J_h\bm{u}-\bm{u}_h\|_{Z_1}$ follows immediately.
%
%
%
%
%

It remains to estimate $\|I_hp-p_h\|_0$. Since $I_hp-p_h\in L^2_0(\Omega)$ when $k\geq 1$, the following inf-sup condition holds (cf. \cite{Girault86})
\begin{align}
\|I_hp-p_h\|_0\leq C \sup_{\bm{v}\in H^1_0(\Omega)^2\backslash \{0\}}\frac{(\nabla \cdot \bm{v},I_hp-p_h)}{\|\bm{v}\|_1}.\label{eq:perror}
\end{align}

An application of the discrete adjoint property, \eqref{eq:scalingG}, integration by parts, \draft{\eqref{eq:bhJ}}, Lemma~\ref{lem:estimate} and the second equation of \eqref{eq:error} implies
\begin{align*}
(\nabla \cdot \bm{v},I_hp-p_h)&=-(\bm{v},\nabla_h (I_hp-p_h))+\sum_{e\in \mathcal{F}_p}(\bm{v}\cdot \bm{n},[I_hp-p_h])_e=-b_h(\bm{v},I_hp-p_h)\\
&=-b_h(J_h\bm{v},I_hp-p_h)=\sqrt{\epsilon}B_h(\Pi_hL-L_h,J_h\bm{v})
+(\alpha(\bm{u}-\bm{u}_h),J_h\bm{v})\\
&\leq C \Big(\sqrt{\epsilon}\|\Pi_h L-L_h\|_{0}\|J_h\bm{v}\|_{Z_2}
+\alpha_{\text{max}}\|\bm{u}-\bm{u}_h\|_0\|J_h\bm{v}\|_0\Big)\\
&\leq C \Big(\sqrt{\epsilon}\|\Pi_h L-L_h\|_{0}
+\alpha_{\text{max}}\|\bm{u}-\bm{u}_h\|_0\Big)\|\bm{v}\|_1,
\end{align*}
which can be combined with \eqref{eq:perror} leads to
\begin{align*}
\|I_hp-p_h\|_0\leq C (\sqrt{\epsilon}\|\Pi_hL-L_h\|_{0}+\alpha_{\text{max}}\|\bm{u}-\bm{u}_h\|_0).
\end{align*}
The proof is complete by combing the above estimates and the approximation properties \eqref{eq:Lerror}-\eqref{eq:uerrorZ}.


\end{proof}

\begin{theorem}(superconvergence).
Assume that $(L,\bm{u},p)\in [H^{k+1}(\Omega)]^{2\times 2}\times H^{k+1}(\Omega)^{2}\times H^{k+1}(\Omega),k\geq 1$ and $(L_h,\bm{u}_h,p_h)\in W^h\times U^h\times P^h$ is the numerical solution of \eqref{eq:SDG-discrete}, then there exists a positive constant $C>0$ independent of the coefficients $\epsilon,\alpha$ such that
\begin{align*}
\|J_h\bm{u}-\bm{u}_h\|_0\leq Ch^{k+2}\Big((1+\sqrt{\epsilon})\|L\|_{k+1}+\alpha_{\text{max}}\|\bm{u}\|_{k+1}+\|p\|_{k+1}\Big).
\end{align*}

\end{theorem}

\begin{proof}

Consider the auxiliary problem
\begin{align*}
-\epsilon \Delta \bm{\phi}+\alpha \bm{\phi}+\nabla \chi =J_h \bm{u}-\bm{u}_h,
\end{align*}
which satisfies the following elliptic regularity estimate (cf. \cite{Girault86,FuQiu19})
\begin{align}
\|\bm{\phi}\|_2+\|\chi\|_1\leq C \|J_h \bm{u}-\bm{u}_h\|_0.\label{eq:regularity}
\end{align}
Let $\sigma=-\sqrt{\epsilon} \nabla \bm{\phi}$, the above equation can be recast into the following first order system
\begin{align*}
\sigma&=-\sqrt{\epsilon}\nabla \bm{\phi}\quad\;\; \mbox{in}\;\Omega,\\
\sqrt{\epsilon}\, \text{div}\, \sigma+\alpha \bm{\phi}+\nabla \chi& =J_h\bm{u}-\bm{u}_h\quad \mbox{in}\;\Omega,\\
\nabla \cdot \bm{\phi}&=0\hspace{1.7cm} \mbox{in}\;\Omega,\\
\bm{\phi}&=\bm{0}\hspace{1.7cm}\mbox{on }\partial \Omega.
\end{align*}
Multiplying the first equation by $L-L_h$, the second equation by $J_h\bm{u}-\bm{u}_h$ and the third equation by $p-p_h$, then we can get from integration by parts
\begin{equation*}
\begin{split}
\|J_h\bm{u}-\bm{u}_h\|_0^2&=\sqrt{\epsilon}(\text{div } \sigma, J_h\bm{u}-\bm{u}_h)+ (\alpha\bm{\phi},J_h\bm{u}-\bm{u}_h)+(\nabla \chi, J_h\bm{u}-\bm{u}_h)+(\sigma, L-L_h)\\
&\;+\sqrt{\epsilon} (\nabla \bm{\phi}, L-L_h)-(\text{div}\,\bm{\phi}, p-p_h)\\
&=-\sqrt{\epsilon}B_h(\sigma,J_h\bm{u}-\bm{u}_h)+(\alpha\bm{\phi},J_h\bm{u}-\bm{u}_h)+b_h^*(\chi, J_h\bm{u}-\bm{u}_h)+(\sigma,L-L_h)\\
&\;+\sqrt{\epsilon}B_h^*(\bm{\phi},L-L_h)+b_h(\bm{\phi},p-p_h),
\end{split}
\end{equation*}
where we use the fact that $L-L_h\in \widehat{W}^h$ and $(\bm{\phi}\cdot \bm{t})\bm{t}\cdot \bm{n}\mid_e=0$ for $e\in \mathcal{F}_p$, so $\sum_{e\in \mathcal{F}_p}([(L-L_h)\bm{n}],(\bm{\phi}\cdot \bm{t})\bm{t})_e=0$.

It then follows from \eqref{eq:error}
\begin{align*}
\|J_h\bm{u}-\bm{u}_h\|_0^2&=-\sqrt{\epsilon}B_h^*(J_h\bm{u}-\bm{u}_h,\sigma-\Pi_h\sigma)
+(L-L_h,\sigma-\Pi_h\sigma)+(\alpha(\bm{u}-\bm{u}_h),\bm{\phi}-J_h\bm{\phi})\\
&\;+(\alpha(J_h\bm{u}-\bm{u}),\bm{\phi}-\pi_h\bm{\phi})
+b_h^*(p-p_h,\bm{\phi}-J_h\bm{\phi})+\sqrt{\epsilon}B_h(L-L_h,\bm{\phi}-J_h\bm{\phi})\\
&\;+b_h(J_h\bm{u}-\bm{u}_h,\chi-I_h\chi)\\
&=(L-L_h,\sigma-\Pi_h\sigma)+(\alpha(\bm{u}-\bm{u}_h),\bm{\phi}-J_h\bm{\phi})
+(\alpha(J_h\bm{u}-\bm{u}),\bm{\phi}-\pi_h\bm{\phi})\\
&\;+b_h^*(p-I_hp,\bm{\phi}-J_h\bm{\phi})
+\sqrt{\epsilon}B_h^*(\bm{\phi}-J_h\bm{\phi},L-\Pi_hL)\\
&=:I_1+I_2+I_3+I_4+I_5.
\end{align*}
We can bound $I_1$ to $I_5$ by \eqref{eq:Lerror}, \eqref{eq:Lerror2}, \eqref{eq:Perror}, \eqref{eq:uerrorZ1} and \eqref{eq:pih} as:
\begin{align*}
I_1&\leq  \|L-L_h\|_0\|\sigma-\Pi_h\sigma\|_0\leq C h\|L-L_h\|_0\|\sigma\|_1,\\
I_2&\leq \alpha_{\text{max}}^{1/2}\|\alpha^{1/2}(\bm{u}-\bm{u}_h)\|_0
\|\bm{\phi}-J_h\bm{\phi}\|_0\leq C \alpha_{\text{max}}^{1/2} h\|\alpha^{1/2}(\bm{u}-\bm{u}_h)\|_0\|\bm{\phi}\|_2,\\
I_3&\leq \alpha_{\text{max}}\|\bm{\phi}-\pi_h\bm{\phi}\|_0\|\draft{J_h\bm{u}-\bm{u}}\|_0\leq C \alpha_{\text{max}} h\|\bm{\phi}\|_1\|J_h\bm{u}-\bm{u}\|_0,\\
I_4&\leq \|\bm{\phi}-J_h\bm{\phi}\|_{Z_1}\normmm{p-I_hp}_{0,h}\leq C h\|\bm{\phi}\|_2\normmm{p-I_hp}_{0,h},\\
I_5&\leq \sqrt{\epsilon}\|L-\Pi_h L\|_{Z'}\|\bm{\phi}-J_h\bm{\phi}\|_{X_1}\leq C \sqrt{\epsilon}h^2\|L-\Pi_h L\|_{Z'}\|\bm{\phi}\|_2.
\end{align*}
Therefore, we have from \eqref{eq:regularity}
\begin{align*}
\|J_h\bm{u}-\bm{u}_h\|_0^2\leq C h^{k+2}\Big((1+\sqrt{\epsilon})\|L\|_{k+1}+\alpha_{\text{max}}\|\bm{u}\|_{k+1}
+\|p\|_{k+1}\Big)\|J_h\bm{u}-\bm{u}_h\|_0,
\end{align*}
which gives the desired estimate.

\end{proof}


\section{Numerical experiments}\label{sec:numerical}

In this section, we present some numerical results to test the performances of the proposed method, where the convergence history for different values of $\epsilon$ are reported. In addition, to confirm that our method can be flexibly applied to fairly general meshes, the highly distorted grids are exploited to display the convergence history.
\begin{example}{\rm
In this example, we consider the exact solution given by
\begin{align*}
\bm{u}&=(\sin(2\pi x)\sin(2\pi y), \sin(2\pi x)\sin(2\pi y))^T,\\
p &= \sin(x)\cos(y)+\sin(1)(\cos(1)-1),
\end{align*}
where we take $\alpha=1$.
}
\end{example}

In all the tables presented below, we use $N_{\text{global}}$ to denote the global degrees of freedom of our final system generated from the discrete problem and D.O.F. stands for the degrees of freedom. We first test our method on the square grids (cf. Figure~\ref{ex1-grid}) and the convergence history for various values of $\epsilon$ with polynomial order $k=1,\cdots,3$ is shown in Table~\ref{table1}-Table~\ref{table4}. It can be observed that the optimal convergence rates of order $k+1$ can be obtained for all the variables $\bm{u}_h,p_h$ with various values of $\epsilon$, which is consistent with our theoretical results in Theorem~\ref{thm:L2}. In addition, we can also observe that the accuracy of $\|\bm{u}-\bm{u}_h\|_0$ remains almost the same whether the Brinkman problem is Darcy-dominating or Stokes-dominating.
We remark that a suboptimal convergence rates can be obtained for $L_h$ when $\epsilon\rightarrow 0$, in which the Brinkman problem collapses to Darcy equation and the explanations of this phenomena can be found in \cite{FuQiu19}.

\begin{figure}[H]
\centering
\scalebox{0.3}{
\includegraphics[width=20cm]{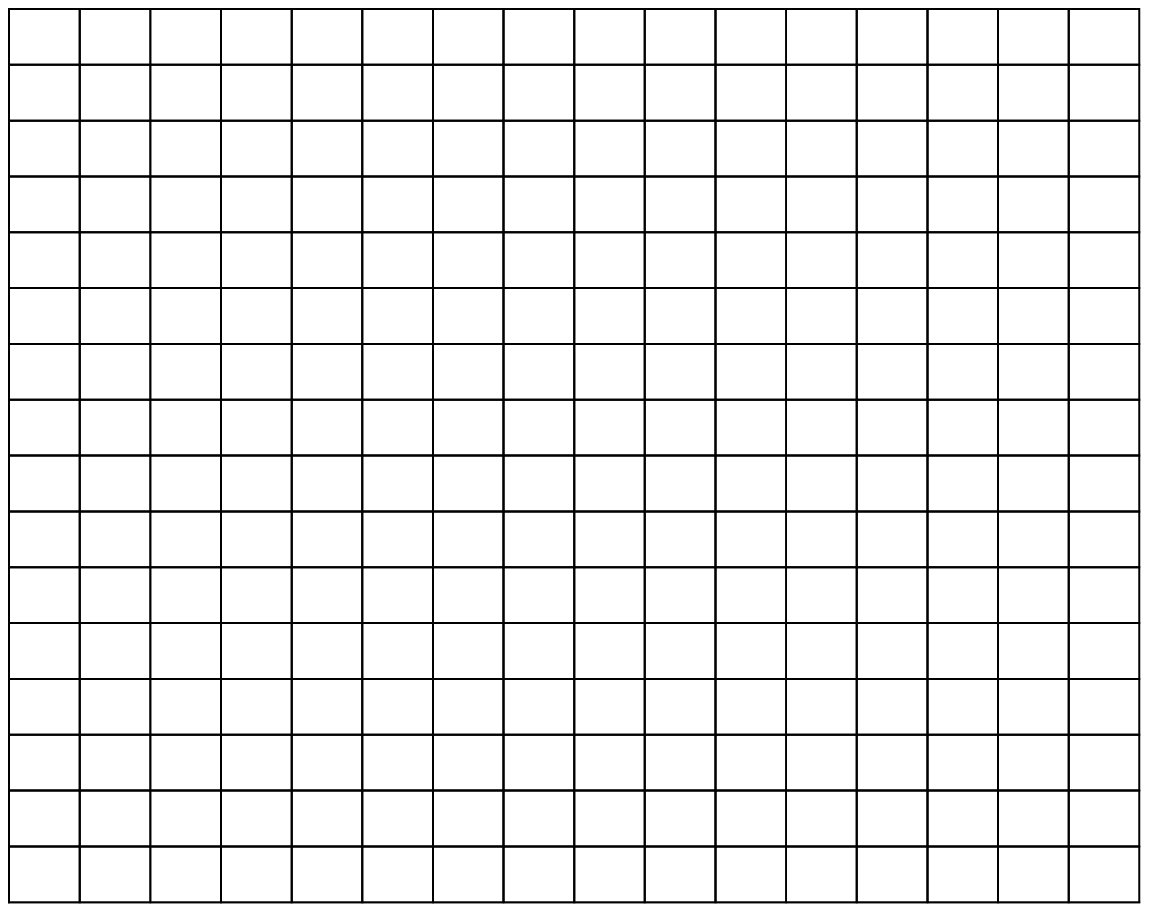}
}
\scalebox{0.3}{
\includegraphics[width=20cm]{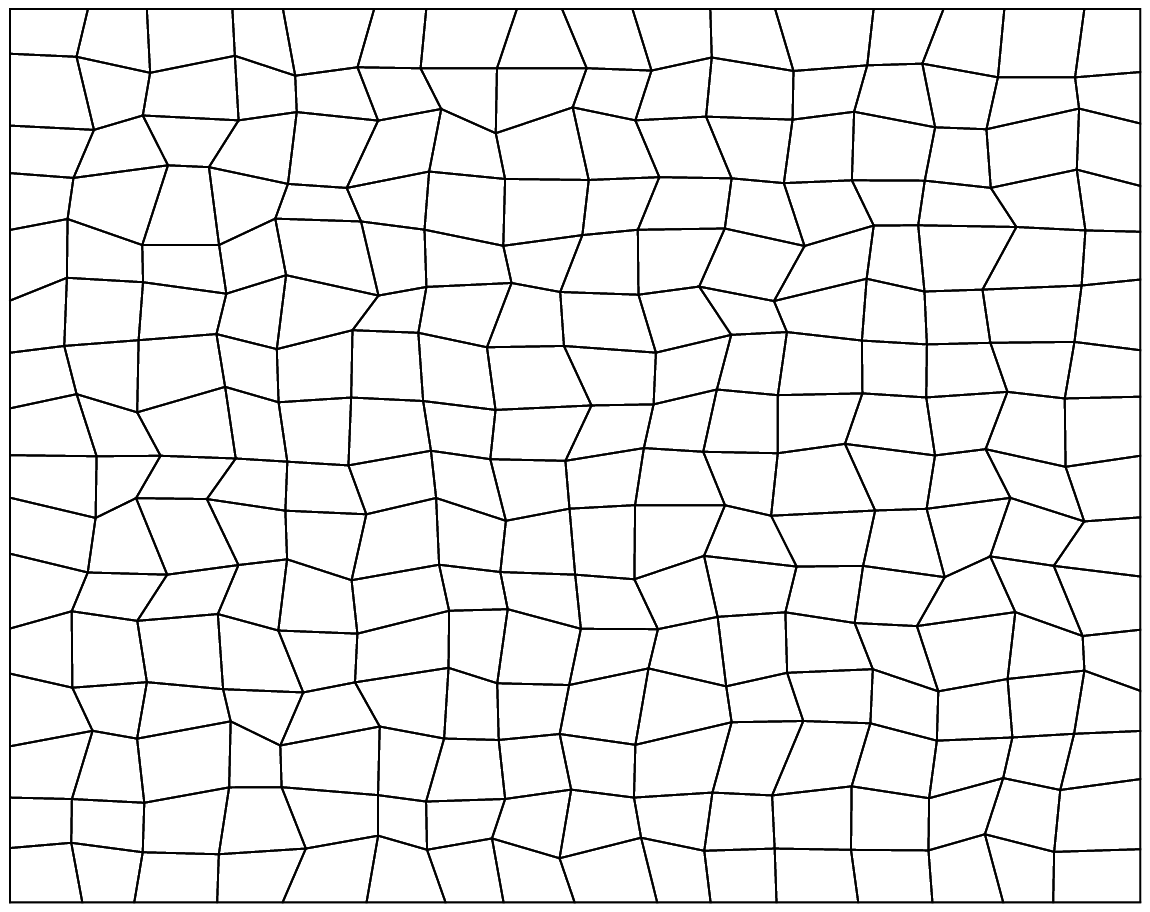}
}
\caption{Schematic of the mesh employed in the simulation: square grids (left) and distorted grids (right).}
\label{ex1-grid}
\end{figure}

\begin{table}[H]
\begin{center}
{\footnotesize
\begin{tabular}{ccc||c c|c c|c c}
\hline
 &Mesh &D.O.F.& \multicolumn{2}{|c|}{$\|\bs{u}-\bs{u}_h\|_{0}$} & \multicolumn{2}{|c|}{$\|L-L_h\|_{0}$} & \multicolumn{2}{|c}{$\|p-p_h\|_{0}$}\\
\hline
$k$ & $h^{-1}$ &$N_{\text{global}}$ & Error & Order & Error& Order& Error & Order \\
\hline
1& 2 &377  & 4.58e-01 &   N/A   &6.36e+00 &  N/A &1.85e+00 &   N/A \\
&  4 &1457 & 1.73e-01 &   1.40  &1.43e+00 & 2.15 &2.97e-01 & 2.63  \\
&  8 &5729  & 4.41e-02 &   1.97  &3.79e-01 & 1.92 &6.98e-02 & 2.09 \\
&  16 &22721  & 1.11e-02 &   1.99  &9.64e-02 & 1.97 &1.62e-02 & 2.10 \\
&  32&90497  & 2.80e-03 &   1.99  &2.42e-02 & 1.99 &3.90e-03 & 2.06  \\
\hline
2& 2 &733 & 1.31e-01 &   N/A     &3.11e-01  &N/A   &1.39e-01 &  N/A  \\
 & 4 &2857 & 1.46e-02  &  3.16    &1.52e-01  &1.03  &3.81e-02 & 1.87\\
 & 8 &11281 & 1.84e-03  &  2.99    &1.97e-02  &2.94  &4.62e-03 & 3.04\\
 & 16 &44833 & 2.30e-04  &  2.99    &2.49e-03  &2.98  &5.76e-04 &3.00\\
 & 32&118753 & 2.88e-05  &  3.00    &3.13e-04  &2.99  &7.20e-05  & 3.00\\
\hline
3& 2 &1201   &1.06e-02   & N/A      &2.84e-01 &N/A   &9.15e-02&N/A \\
&  4 &4705   &1.37e-03   &2.95      &1.42e-02 &4.32  &3.88e-03&4.55\\
&  8 &18625  &8.64e-05   &3.98      &9.09e-04 &3.96  &2.42e-04&4.00\\
&  16 &74113  &5.41e-06   &3.99      &5.73e-05 &3.99  &1.48e-05&4.03\\
& 32&295681 &3.38e-07   &3.99      &3.59e-06 &3.99  &9.12e-07&4.01\\
\hline
\end{tabular}}
\caption{Convergence history for $\epsilon = 10^{0}$ on square grids.}
\label{table1}
\end{center}
\end{table}

\begin{table}[H]
\begin{center}
{\footnotesize
\begin{tabular}{ccc||c c|c c|c c}
\hline
 &Mesh &D.O.F.& \multicolumn{2}{|c|}{$\|\bs{u}-\bs{u}_h\|_{0}$} & \multicolumn{2}{|c|}{$\|L-L_h\|_{0}$} & \multicolumn{2}{|c}{$\|p-p_h\|_{0}$}\\
\hline
$k$ & $h^{-1}$&$N_{\text{global}}$ & Error & Order & Error & Order &  Error & Order \\
\hline
1& 2 &377 & 2.96e-01 &   N/A    &5.89e-01 &  N/A &1.91e-02 &   N/A \\
&  4 &1457 & 1.71e-01 &   0.79   &1.37e-01 & 2.10 &4.22e-03 & 2.18  \\
&  8 &5729 & 4.39e-02 &   1.96   &3.73e-02 & 1.88 &9.50e-04 & 2.15 \\
&  16 &22721 & 1.11e-02 &   1.98   &9.59e-03 & 1.96 &2.20e-04 & 2.10 \\
&  32&90497 & 2.78e-03 &   1.99   &2.42e-03 & 1.99 &5.33e-05 & 2.05  \\
\hline
2& 2 &733  & 1.30e-01 &   N/A     &3.47e-02  &N/A   &2.85e-03 &  N/A  \\
 & 4 &2857  & 1.44e-02  &  3.17    &1.46e-02  &1.24  &3.96e-04 & 2.85\\
 & 8 &11281  & 1.83e-03  &  2.97    &1.96e-03  &2.91  &4.67e-05 & 3.08\\
 & 16 &44833  & 2.30e-04  &  2.99    &2.49e-04  &2.97  &5.80e-06 & 3.00\\
 & 32&118753 & 2.88e-05  &  2.99    &3.13e-05  &2.99  &7.25e-07 & 3.00\\
\hline
3& 2&1201  &9.43e-03   & N/A      &2.67e-02 &N/A   &8.98e-04&N/A \\
&  4&4705  &1.36e-03   &2.79      &1.39e-03 &4.26  &4.07e-05&4.46\\
&  8&18625  &8.63e-05   &3.97      &9.09e-05 &3.94  &2.45e-06&4.05\\
&  16&74113  &5.41e-06   &3.99      &5.72e-06 &3.98  &1.48e-07&4.04\\
& 32&295681  &3.38e-07   &3.99      &3.59e-07 &3.99  &9.13e-09&4.02\\
\hline
\end{tabular}}
\caption{Convergence history for $\epsilon = 10^{-2}$ on square grids.}
\end{center}
\end{table}

\begin{table}[H]
\begin{center}
{\footnotesize
\begin{tabular}{ccc||c c|c c|c c}
\hline
 &Mesh &D.O.F.& \multicolumn{2}{|c|}{$\|\bs{u}-\bs{u}_h\|_{0}$} & \multicolumn{2}{|c|}{$\|L-L_h\|_{0}$} & \multicolumn{2}{|c}{$\|p-p_h\|_{0}$}\\
\hline
$k$ & $h^{-1}$ &$N_{\text{global}}$ & Error & Order & Error & Order &  Error & Order \\
\hline
1& 2 &377  & 1.72e-01 &   N/A    &5.63e-02 &  N/A &1.23e-02 &   N/A \\
&  4 &1457  & 1.73e-01 &   0.008  &1.30e-02 & 2.11 &2.88e-03 & 2.09  \\
&  8 &5729  & 4.41e-02 &   1.97   &3.69e-03 & 1.81 &6.24e-04 & 2.20 \\
&  16 &22721  & 1.11e-02 &   1.98   &1.02e-03 & 1.85 &1.48e-04 & 2.07 \\
&  32&90497  & 2.78e-03 &   1.99   &2.59e-04 & 1.98 &3.65e-05 & 2.02  \\
\hline
2& 2 &733 & 1.29e-01 &   N/A     &5.19e-03  &N/A   &1.95e-03 &  N/A  \\
 & 4 &2857 & 1.37e-02  &  3.24    &1.43e-03  &1.86  &1.36e-04 & 3.84\\
 & 8 &11281 & 1.78e-03  &  2.95    &2.33e-04  &2.62  &9.87e-06 & 3.78\\
 & 16 &44833 & 2.26e-04  &  2.98    &3.19e-05  &2.87  &8.42e-07 & 3.55\\
 & 32&118753 & 2.85e-05  &  2.98    &3.62e-06  &3.13  &8.79e-08 & 3.26\\
\hline
3& 2 &1201 &5.86e-03   & N/A      &2.15e-03 &N/A   &1.03e-04&N/A \\
&  4 &4705 &1.33e-03   &2.14      &1.56e-04 &3.78  &9.15e-06&3.50\\
&  8 &18625 &8.46e-05   &3.97      &1.20e-05 &3.70  &3.04e-07&4.91\\
&  16 &74113 &5.34e-06   &3.98      &7.42e-07 &4.01  &1.01e-08&4.91\\
& 32&295681 &3.37e-07   &3.99      &4.06e-08 &4.19  &3.46e-010&4.87\\
\hline
\end{tabular}}
\caption{Convergence history for $\epsilon = 10^{-4}$ on square grids.}
\end{center}
\end{table}

\begin{table}[H]
\begin{center}
{\footnotesize
\begin{tabular}{ccc||c c|c c|c c}
\hline
 &Mesh &D.O.F.& \multicolumn{2}{|c|}{$\|\bs{u}-\bs{u}_h\|_{0}$} & \multicolumn{2}{|c|}{$\|L-L_h\|_{0}$} & \multicolumn{2}{|c}{$\|p-p_h\|_{0}$}\\
\hline
$k$ & $h^{-1}$ &$N_{\text{global}}$ & Error & Order & Error & Order &  Error & Order \\
\hline
1& 2 &377  & 1.69e-01 &   N/A    &5.60e-04 &  N/A &1.23e-02 &   N/A \\
&  4 &1457  & 1.73e-01 &   0.03   &1.30e-04 & 2.10 &2.87e-03 & 2.09  \\
&  8 &5729  & 4.42e-02 &   1.97   &3.83e-05 & 1.77 &6.22e-04 & 2.20 \\
&  16 &22721  & 1.11e-02 &   1.98   &1.26e-05 & 1.60 &1.48e-04 & 2.07 \\
&  32&90497  & 2.79e-03 &   1.99   &5.11e-06 & 1.30 &3.64e-05 & 2.02  \\
\hline
2& 2 &733 & 1.30e-01 &   N/A     &5.32e-05  &N/A   &1.94e-03 &  N/A  \\
 & 4 &2857 & 1.37e-02  &  3.24    &1.48e-05  &1.84  &1.35e-04 & 3.84\\
 & 8 &11281 & 1.78e-03  &  2.94    &3.02e-06  &2.30  &9.76e-06 & 3.79\\
 & 16 &44833 & 2.24e-04  &  2.98    &7.18e-07  &2.07  &8.30e-07 & 3.55\\
 & 32&178753 & 2.81e-05  &  2.99    &1.78e-07  &2.01  &8.71e-08 & 3.25\\
\hline
3& 2 &1201 &5.81e-03   & N/A      &2.14e-05 &N/A   &1.00e-04&N/A \\
&  4 &4705 &1.34e-03   &2.11      &1.78e-06 &3.58  &9.02e-06&3.48\\
&  8 &18625 &8.57e-05   &3.97      &1.88e-07 &3.25  &2.94e-07&4.94\\
&  16 &74113 &5.38e-06   &3.99      &2.17e-08 &3.11  &9.41e-09&4.96\\
& 32&295681 &3.37e-07   &3.99      &2.63e-09 &3.04  &3.10e-010&4.92\\
\hline
\end{tabular}}
\caption{Convergence history for $\epsilon = 10^{-8}$ on square grids.}
\label{table4}
\end{center}
\end{table}

Next we test the performances of our method on highly distorted grids as shown in Figure~\ref{ex1-grid} and the results are provided in Table~\ref{table5}-Table~\ref{table8}. As expected, optimal convergence rates of order $k+1$ can be observed for $\bm{u}_h$ and $p_h$ with various values of $\epsilon$. In addition, a suboptimal convergence rates can be achieved for $L_h$ when the Brinkman problem becomes Darcy-dominating. The results presented here indicate that our method is robust to the rough grids.

\begin{table}[H]
\begin{center}
{\footnotesize
\begin{tabular}{ccc||c c|c c|c c}
\hline
 &Mesh &D.O.F.& \multicolumn{2}{|c|}{$\|\bs{u}-\bs{u}_h\|_{0}$} & \multicolumn{2}{|c|}{$\|L-L_h\|_{0}$} & \multicolumn{2}{|c}{$\|p-p_h\|_{0}$}\\
\hline
$k$ & $h^{-1}$ &$N_{\text{global}}$ & Error & Order & Error & Order &  Error & Order \\
\hline
1& 2 &377   & 5.48e-01 &   N/A    &6.37e+00 &  N/A &2.05e+00 &   N/A \\
&  4 &1457  & 1.97e-01 &   1.47   &1.71e+00 & 1.89 &3.89e-01 & 2.39  \\
&  8 &5729  & 5.40e-02 &   1.87   &4.23e-01 & 2.02 &8.44e-02 & 2.20 \\
&  16 &22721  & 1.32e-02 &  2.02   &1.14e-01 & 1.89 &2.15e-02 & 1.97 \\
&  32&90497  & 3.37e-03 &  1.97   &2.89e-02 & 1.98 &5.27e-03 & 2.03  \\
\hline
2& 2 &733   & 1.79e-01 &   N/A    &6.06e-01  &N/A   &1.85e-01 &  N/A  \\
 & 4 &2857  &2.05e-02  &  3.12    &2.63e-01  &1.55  &5.38e-02 & 1.78\\
 & 8 &11281 &2.79e-03  &  2.87    &2.83e-02  &2.86  &7.11e-03 &2.92\\
 & 16 &44833 &3.37e-04  &  3.05    &3.57e-03  &2.98  &8.19e-04 &3.11\\
 & 32&178753&4.47e-05  &  2.91    &4.56e-04  &2.96  &1.07e-04 &2.93\\
\hline
3& 2 &1201 &1.39e-02   & N/A      &3.24e-01 &N/A   &1.00e-01&N/A \\
&  4 &4705 &2.42e-03   &2.53      &2.24e-02 &3.85  &6.07e-03&4.04\\
&  8 &18625 &1.62e-04  &3.90      &1.55e-03 &3.85  &3.97e-04&3.93\\
&  16 &74113 &1.03e-05  &3.97      &1.04e-04 &3.90  &2.60e-05&3.93\\
& 32&295681 &6.99e-07 &3.88      &6.69e-06 &3.96  &1.66e-06&3.97\\
\hline
\end{tabular}}
\caption{Convergence history for $\epsilon = 1$ on distorted grids.}
\label{table5}
\end{center}
\end{table}

\begin{table}[H]
\begin{center}
{\footnotesize
\begin{tabular}{ccc||c c|c c|c c}
\hline
 &Mesh &D.O.F.& \multicolumn{2}{|c|}{$\|\bs{u}-\bs{u}_h\|_{0}$} & \multicolumn{2}{|c|}{$\|L-L_h\|_{0}$} & \multicolumn{2}{|c}{$\|p-p_h\|_{0}$}\\
\hline
$k$ & $h^{-1}$ &$N_{\text{global}}$ & Error & Order & Error & Order &  Error & Order \\
\hline
1& 2 &377   & 3.95e-01 &   N/A    &5.86e-01 &  N/A &2.04e-02 &   N/A \\
&  4 &1457  & 1.99e-01 &   1.01   &1.57e-01 & 1.95 &5.38e-03 & 1.97  \\
&  8 &5729  & 5.16e-02 &   1.97   &4.54e-02 & 1.81 &1.18e-03 & 2.21 \\
&  16 &22721  & 1.33e-02 &  1.97   &1.17e-02 & 1.97 &2.85e-04 & 2.06 \\
&  32&90497  & 3.39e-03 &  1.97   &2.93e-03 & 1.99 &7.03e-05 & 2.02  \\
\hline
2& 2 &733   & 1.62e-01 &   N/A    &5.68e-02  &N/A   &4.23e-03 &  N/A  \\
 & 4 &2857  &1.75e-02  &  3.27    &1.84e-02  &1.66  &5.38e-04 & 3.02\\
 & 8 &11281 &2.73e-03  &  2.70    &2.63e-03  &2.83  &6.97e-05 &2.98\\
 & 16 &44833 &3.57e-04  &  2.94    &3.70e-04  &2.84  &8.95e-06 &2.97\\
 & 32&178753&4.46e-05  &  3.00    &4.50e-05  &3.04  &1.07e-06 &3.07\\
\hline
3& 2 &1201 &1.34e-02   & N/A      &3.07e-02 &N/A   &9.85e-04&N/A \\
&  4 &4705 &2.37e-03   &2.55      &2.11e-03 &3.86  &6.20e-05&3.98\\
&  8 &18625 &1.52e-04  &3.99      &1.55e-04 &3.77  &4.22e-06&3.88\\
&  16&74113 &9.98e-06  &3.94      &1.03e-05 &3.91  &2.61e-07&4.01\\
& 32&295681 &6.53e-07 &3.94      &6.42e-07 &4.00  &1.62e-08&4.00\\
\hline
\end{tabular}}
\caption{Convergence history for $\epsilon = 10^{-2}$ on distorted grids.}
\end{center}
\end{table}

\begin{table}[H]
\begin{center}
{\footnotesize
\begin{tabular}{ccc||c c|c c|c c}
\hline
 &Mesh &D.O.F.& \multicolumn{2}{|c|}{$\|\bs{u}-\bs{u}_h\|_{0}$} & \multicolumn{2}{|c|}{$\|L-L_h\|_{0}$} & \multicolumn{2}{|c}{$\|p-p_h\|_{0}$}\\
\hline
$k$ & $h^{-1}$ &$N_{\text{global}}$ & Error & Order & Error & Order &  Error & Order \\
\hline
1& 2 &377   & 2.66e-01 &   N/A    &5.67e-02 &  N/A &1.39e-02 &   N/A \\
&  4 &1457  & 1.95e-01 &   0.46   &1.62e-02 & 1.86 &3.34e-03 & 2.11  \\
&  8 &5729  & 5.03e-02 &   1.98   &4.72e-03 & 1.80 &7.52e-04 & 2.17 \\
&  16 &22721  &1.33e-02 &  1.94    &1.52e-03 & 1.64 &1.84e-04 & 2.04 \\
&  32&90497  & 3.36e-03 &  1.99   &3.71e-04 & 2.04 &4.50e-05 & 2.03  \\
\hline
2& 2 &733   &1.32e-01 &   N/A    &6.19e-03  &N/A   &1.96e-03 &  N/A  \\
 & 4 &2857  &2.06e-02  &  2.74    &2.73e-03  &1.20  &2.45e-04 & 3.05\\
 & 8 &11281 &2.97e-03  &  2.82    &4.29e-04  &2.70  &1.85e-05 &3.76\\
 & 16 &44833 &3.41e-04  &  3.13    &5.47e-05  &2.98  &1.36e-06 &3.78\\
 & 32&178753&4.37e-05  &  2.97    &6.23e-06  &3.14  &1.34e-07 &3.35\\
\hline
3& 2 &1201 &1.15e-02   & N/A      &3.15e-03 &N/A   &1.59e-04&N/A \\
&  4 &4705 &2.89e-03   &2.42      &3.06e-04 &3.42  &1.76e-05&3.22\\
&  8 &18625 &1.45e-04  &3.95      &2.28e-05 &3.77  &6.27e-07&4.84\\
&  16 &74113 &1.02e-05  &3.83      &1.74e-06 &3.73  &2.36e-08&4.75\\
& 32&295681 &6.62e-07 &3.96      &8.98e-08 &4.28  &8.16e-010&4.86\\
\hline
\end{tabular}}
\caption{Convergence history for $\epsilon = 10^{-4}$ on distorted grids.}
\end{center}
\end{table}

\begin{table}[H]
\begin{center}
{\footnotesize
\begin{tabular}{ccc||c c|c c|c c}
\hline
 &Mesh &D.O.F.& \multicolumn{2}{|c|}{$\|\bs{u}-\bs{u}_h\|_{0}$} & \multicolumn{2}{|c|}{$\|L-L_h\|_{0}$} & \multicolumn{2}{|c}{$\|p-p_h\|_{0}$}\\
\hline
$k$ & $h^{-1}$ &$N_{\text{global}}$ & Error & Order & Error & Order &  Error & Order \\
\hline
1& 2 &377   & 2.12e-01 &   N/A    &5.76e-04 &  N/A &1.28e-02 &   N/A \\
&  4 &1457  & 1.89e-01 &   0.16   &1.51e-04 & 1.98 &3.47e-03 & 1.94  \\
&  8 &5729  & 5.03e-02 &   1.93   &5.09e-05 & 1.59 &7.66e-04 & 2.20 \\
&  16 &22721  &1.33e-02 &  1.93   &1.86e-05 & 1.46 &1.82e-04 & 2.08 \\
&  32&90497  & 3.36e-03&  1.99   &8.48e-06 & 1.14 &4.52e-05 & 2.01  \\
\hline
2& 2 &733   &1.44e-01 &   N/A    &6.85e-05  &N/A   &2.33e-03 &  N/A  \\
 & 4&2857  &1.71e-02  &  3.13    &2.03e-05  &1.78  &1.90e-04 & 3.68\\
 & 8&11281 &2.77e-03  &  2.65    &5.02e-06  &2.03  &1.77e-05 &3.46\\
 & 16&44833 &3.24e-04  &  3.11    &1.17e-06  &2.11  &1.31e-06 &3.77\\
 & 32&178753&4.18e-05  &  2.96    &2.89e-07  &2.02  &1.30e-07 &3.34\\
\hline
3& 2 &1201 &1.10e-02   & N/A      &3.34e-05 &N/A   &1.44e-04&N/A \\
&  4 &4705 &2.11e-03   &2.38      &2.88e-06 &3.58  &1.57e-05&3.24\\
&  8 &18625 &1.72e-04  &3.64      &5.40e-07 &2.43  &6.99e-07&4.52\\
&  16 &74113 &9.65e-06  &4.17      &5.04e-08 &2.43  &1.97e-08&5.16\\
& 32&295681 &6.34e-07 &3.93      &6.41e-09 &2.98  &6.91e-010&4.85\\
\hline
\end{tabular}}
\caption{Convergence history for $\epsilon = 10^{-8}$ on distorted grids.}
\label{table8}
\end{center}
\end{table}

\section{Conclusion}

In this paper, we propose a new SDG method which is robust for \draft{Brinkman} problem in Stokes and Darcy limits. In addition, the method is highly flexible by allowing fairly general meshes and hanging \draft{nodes} can be automatically incorporated in the construction of the method. Moreover, no numerical flux or stabilization term is required in our method. All these distinctive features make the proposed method a good candidate in practical applications. Several numerical experiments are reported to verify the robustness and accuracy of the proposed method.

\section*{Acknowledgements}

The research of Eric Chung is partially supported by the Hong Kong RGC General Research Fund (Project numbers 14304217 and 14302018), CUHK Faculty of Science Direct Grant 2018-19 and NSFC/RGC Joint Research Scheme (Project number HKUST620/15).

\end{document}